\documentclass[10 pt,reqno]{amsart}
\pdfoutput=1
\usepackage{amsmath,amsthm,amsfonts,amssymb,mathrsfs,bm,graphicx,stmaryrd,dsfont,esvect}
\usepackage[colorlinks=true,linkcolor=blue]{hyperref}
\usepackage[usenames,dvipsnames]{color}
\numberwithin{equation}{section}
\usepackage[letterpaper,hmargin=1in,vmargin=1in]{geometry}
\parskip 	\smallskipamount

\newtheorem{theorem}{Theorem}[section]
\newtheorem{lemma}[theorem]{Lemma}

\newtheorem{proposition}[theorem]{Proposition}

\theoremstyle{definition}
\newtheorem{remark}[theorem]{Remark}

\newtheorem{definition}[theorem]{Definition}

\begin{document}
\title[Deviations in stationary LPP]{Moderate deviation and exit time
estimates for stationary Last
Passage Percolation}
\author{Manan Bhatia}
\address{Manan Bhatia, Department of Mathematics, Indian Institute of Science, Bangalore, India}
\email{mananbhatia@iisc.ac.in}
\date{}
\maketitle
\begin{abstract} 
	We consider planar stationary exponential Last Passage Percolation in the positive
	quadrant with boundary weights. For $\rho\in (0,1)$ and points
	$v_N=( (1-\rho)^2 N, \rho^2 N)$ going to infinity along the
	characteristic direction, we establish right tail estimates with the optimal
	exponent for the exit
	time of the geodesic, along with optimal exponent estimates for the upper tail
	moderate deviations for the passage time. For the case $\rho=\frac{1}{2}$ in the stationary model, we
	establish the lower bound estimate with the optimal exponent for the
	lower tail of the passage time. Our arguments are based on moderate
	deviation estimates for point-to-point and point-to-line exponential Last Passage
	Percolation which are obtained via random matrix estimates.
\end{abstract}
\section{Introduction}
	The planar exponential Last Passage Percolation (LPP) model is an
	important and canonical integrable model in the (1+1)-dimensional KPZ
	universality class. The model has been mainly studied for three initial conditions-- step, flat
	and stationary. Currently, there are two main approaches for analysing
	these models-- the first relying on using the random matrix connections
	for the models with the step and flat initial conditions to obtain
	concentration for the passage time \cite{BSS14,BGHK19,BGZ19,BSS19}. The second approach relies on using
	duality along with the Burke property for the stationary initial
	condition \cite{BCS06,Pim16,SS19}.

	For the stationary initial condition, the exit time is an important
	quantity which has been used to establish the correct order of the
	variance of the stationary passage time along the
	characteristic line \cite{BCS06}, along with optimal estimates for
	the coalescence time of two semi-infinite geodesics in exponential
	LPP \cite{Pim16,SS19}. Until very recently, only suboptimal tail estimates were
	available for the exit time \cite{BCS06}. Some estimates for the exit time
	have also been obtained using Fredholm determinantal formulae
	\cite{FO18,FO18+}, and the lower bound with the optimal exponent for the
exit time is known \cite{Sep17, BS10}. Though there are exact
	correspondences to the eigenvalues of certain random matrices for the
	passage time in the case of the step and flat initial conditions
	\cite{Jo00,BGHK19,LR10}, no
	such correspondence is known in the case of the stationary initial
	condition which makes it difficult to directly use inputs from the random matrix
	literature for its
	analysis. 
	
		In this paper, in Theorem \ref{tails_upper_upper} and
	Theorem \ref{tails_lower_lower}, we show that one can use the known
	concentration estimates for
	point-to-point and point-to-line LPP originating from the random matrix
	connections \cite{LR10, BSS14, BGHK19} to obtain optimal moderate deviation concentration estimates
	for upper and lower tails of the passage time in the stationary model. 
	In Theorem \ref{exit_time}, we obtain optimal exponent
	right tail estimates for the exit
	time by using similar techniques. As a matter of fact, a possible proof of the optimal
	exponents via the transversal fluctuation estimates proved in
	\cite{BSS19} using the moderate deviation estimates coming from the random
	matrix connections, together with duality and Busemann functions was
	indicated in \cite{SS19}; we, however, provide a direct proof using the
	moderate deviation estimates without appealing to duality and transversal
	fluctuation estimates. 
	
	As we were finishing up this paper, the paper \cite{EJS20} was posted on the
	arXiv where the optimal exponent right tail estimates for the exit time, as
	well as the optimal exponent upper tail estimates for the passage time for the
	stationary initial condition are obtained. The proof however, is different from ours-- the
	approach is based on obtaining an exact formula for the moment
	generating function for the stationary passage time and then using it to
	obtain the other results. In contrast, our approach proceeds by using the
	random matrix estimates for the point-to-point and the point-to-line
	passage times and then using it to obtain the results for the stationary
	model. In a broader context, the two approaches mentioned at the
	beginning of the introduction have sometimes led to parallel results
	(\cite{BSS19}, \cite{Pim16,SS19} and \cite{BHS18}, \cite{BBS19}), and this
	paper is also in the same spirit.

	The estimates for the lower bound on the upper tail and the upper bound
	on the lower tail of the stationary passage time are immediate from
	comparison with the point-to-point model. For the special case of
	$\rho=\frac{1}{2}$ in the stationary model, we were also able to
	obtain the estimate with the optimal exponent for the lower bound on the lower tail which
	is not available in the literature so far. As the reader will see, our
	proofs only
	depend on the moderate deviation
	estimates from random matrix connections together with the strict concavity of the shape function
	$(\sqrt{x}+\sqrt{y})^2$, and the same proofs are expected to work for other models where such estimates are known, e.g. stationary versions of Poissonian and Geometric LPP.

\paragraph{\textbf{Outline of the paper}} In Section 2, we give the precise
definitions of the models that we are working with and provide the
statements of our main results. In Section 3, we state known results
relating to deviation estimates for the point-to-point and point-to-line
passage times; we will be using these in our arguments. In Section 4, we prove Theorem
\ref{exit_time}, the tight upper bound on the upper tail of the stationary LPP
exit time. The upper and lower tail estimates for the stationary last passage time-- Theorem
\ref{tails_upper_upper} and Theorem \ref{tails_lower_lower} are proved in
Sections 5 and 6 respectively.

	\paragraph{\textbf{Acknowledgements}} The author thanks Riddhipratim Basu for useful
	discussions and valuable encouragements. The author was supported by the
	KVPY fellowship from the Government of India, along with the LTVSP
	program at ICTS, Bangalore.
	\section{Model definitions and main results}
\begin{definition}[Planar exponential LPP]
	\label{step}
	Define a random field
	\begin{displaymath}
		\omega=\left\{ \omega_v : v\in \mathbb{Z}^2 \right\}.
	\end{displaymath}
	where the $\omega_v$ are i.i.d.\ $\mathrm{exp}(1)$ random variables. Given $u,v\in
	\mathbb{Z}^2$ with $u\leq v$ (that is, $u$ is coordinate-wise smaller
	than $v$), for any
	up-right path $\gamma$ from $u$ to $v$, define the weight of the path
	$l(\gamma)$ as 
	\begin{displaymath}
		l(\gamma)=\sum_{w\in\gamma} \omega_w.
	\end{displaymath}
	For any $u\leq v$, define the point-to-point passage time $G(u,v)$ by the maximum
	of $l(\gamma)$ over all up-right paths from $u$ to $v$, and define it to
	be $-\infty$ otherwise. Call the a.s.
	unique path attaining the maximum as the geodesic from $u$ to $v$. In a
	similar manner, we can also define point-to-line passage times. To
	simplify notation later, we define the point-to-line passage time from a
	point to the line $\{x+y=0\}$ with a weight function.
	Namely, if $T$ is a possibly random weight function on the line
	$\{x+y=0\}$, and $v$ is a point above the line, then one can define the
	point-to-line passage time with initial condition $T$ as
	\begin{equation}
		G_T(v)= \max_{u\in \{x+y=0\}}\{T(u)+G(u,v)\}.
		\label{point_to_line_defn}
	\end{equation}
	Note that if $T$ is identically $0$, then we denote the
	corresponding point-to-line passage time by $G_0(\cdot)$. Also note
	that we
	will be using the notation $\mathbb{P}(\cdot)$ when denoting probabilities for
	this model.
\end{definition}

Note that though we defined the point-to-point and point-to-line passage
times for the environment in Definition \ref{step}, these quantities can be
similarly defined for other environments too. We now come to the stationary
LPP model. The Totally Asymmetric Simple Exclusion Process (TASEP) started from a
given initial configuration of particles and holes can be realized as a
corresponding LPP model, and the stationary LPP model is the one
corresponding to the TASEP started from a stationary distribution. Note
that the TASEP and hence the stationary LPP model has a one parameter family
of invariant measures parametrised by the particle density
$\rho \in [0,1]$. We will be using two different
representations of the stationary model. For clarity, we define them
separately. Note that we use $\mathbf{0}$ for $(0,0)$.
\begin{definition}[Boundary representation of stationary LPP with density
	$\rho\in (0,1)$]
	\label{first_representation}
	This model was introduced in \cite{BCS06} and is defined on the positive
	quadrant $\mathbb{Z}^2_{\geq 0}$. Let $\mathrm{e}_1,\mathrm{e}_2$ denote
	the unit vectors along the $x$ and $y$ axes respectively.
	Define a random field
		\begin{displaymath}
			\omega'=\left\{ \omega_v' : v\in \mathbb{Z}_{\geq 0}^2 \right\},
	\end{displaymath}
	where the $\omega_v'$ are independent random variables such that
	$\omega_{\mathbf{0}}'= 0$ and
	$\omega_v'\sim \mathrm{exp}(1)$ for all $v\in \mathbb{Z}^2_{>0}$. For the
	coordinate axes, we have that $\omega_{i\rm{e}_1}'\sim \mathrm{exp}(1-\rho)$ and
	$\omega_{j\rm{e}_2}'\sim \mathrm{exp}(\rho)$ for all $i,j\geq 1$. For any
	$v\in \mathbb{Z}^2_{\geq 0}$, define the stationary
	last passage time
	time $G^1_{\mathrm{stat}}(v)$ as the maximum of $l(\gamma)$ over all up-right paths from
	$\mathbf{0}$ to $v$. Here, $l(\gamma)$ is defined as earlier by using the
	weights from the environment $\omega'$. Call the a.s.
	unique path attaining the maximum as the stationary geodesic to $v$. When denoting probabilities for this model, we
	will use the notation $\mathbb{P}^{\rho}(\cdot)$.
\end{definition}

\begin{definition}[Point-to-line representation of stationary LPP with density
	$\rho
	\in (0,1)$] 
	\label{second_representation}
	We take this model from \cite{FO18}. The model is defined on the upper-right half-plane $\{x+y\geq
	0\}\subseteq \mathbb{Z}^2$. Define the random field
	\begin{displaymath}
		\omega''=\left\{ \omega_v'': v\in \mathbb{Z}^2 \right\},
	\end{displaymath}
	where the $\omega_v''$ are independent random variables such that
	$\omega_v''\sim \mathrm{exp}(1)$ for all $v\in \left\{ x+y>0
	\right\}$, and $\omega_v''=0$ otherwise. For $v\in \{x+y=0\}$, let $\tau_v,\psi_v$ be
	random variables independent of each other and $\omega''$ with the
	marginals $\tau_v\sim \mathrm{exp}(1-\rho)$ and $\psi_v\sim
	\mathrm{exp}(\rho)$. We now define a weight function on the line
	$\{x+y=0\}$. Given $v\in \{x+y=0\}$ with
	$v=(t,-t)$, define
	\begin{displaymath}
		T(v)=
		\begin{cases}
			0, & \text{for } t=0\\
			\sum_{s=1}^{t}\left( \tau_{(s,-s)}-\psi_{(s,-s)} \right),
			&\text{for } t>0\\
			-\sum_{s=t}^{-1}\left( \tau_{(s,-s)}-\psi_{(s,-s)} \right),
			&\text{for } t<0
		\end{cases}
	\end{displaymath}
	For $v\in \{x+y\geq 0\}$, define 
the the last passage time
$G^2_{\mathrm{stat}}(v)$ by
	\begin{equation}
		\label{repn_pt_line}
		G^2_{\mathrm{stat}}(v)=G_T(v)=\max_{u\in \{x+y=0\}}\{T(u)+G(u,v)\},
	\end{equation}
	where the point-to-point passage times $G$ are computed using
	$\omega''$. Call the a.s.
	unique path attaining the maximum (starting from the line
	$\{x+y=0\}$ and ending at $v_N$) as the stationary geodesic to $v$. When denoting probabilities for this model, we will use the notation
	$\overline{\mathbb{P}}^{\rho}(\cdot)$.
\end{definition}

The two representations are known to be equivalent, and a proof is
provided in the appendix. We will mostly be working with the boundary representation of the stationary
model, but the point-to-line representation will be used in the proof of
Theorem \ref{tails_lower_lower} to obtain the lower
bound for the lower tail for the stationary passage time for
$\rho=\frac{1}{2}$.

	\paragraph{\textbf{Notational comments}} We try to use the letters $C,c$ for constants in all the
Lemmas, Propositions and Theorems to prevent cluttering due to an overuse of
subscripts; we do not mean that all of the results are true with the same
constants. Regarding boldface letters, if $x>0$,
then $\mathbf{x}$ denotes $(x,0)$ and $\mathbf{x}^{\uparrow}$ denotes
$(x,1)$. If $x<0$ then $\mathbf{x}$ denote $(0,-x)$ and
$\mathbf{x}^{\uparrow}$ denotes $(1,-x)$. Finally, $\mathbf{0}$ denotes $(0,0)$. For a real variable
$t\neq 0$, we use the notation $\mathfrak{t}$ for the point $(-t,t)$. We
denote $(n,n)$ by $\vv{n}$. By
$\rm{e}_1,\rm{e}_2$, we denote the unit vectors along the $x$ and $y$ axes
respectively in the plane. To prevent cluttering due to ceiling and floor
signs, we do not worry about rounding issues; our arguments are insensitive to
them and remain unchanged.

	\subsection{Statements of results}
	The first result concerns right tail estimates for the exit time in stationary
	LPP. Let $v_N=\left( (1-\rho)^2 N, \rho^2 N \right)$ be a general point on the
characteristic line (see \cite{SS19}) for the stationary model of density
$\rho$. In the setting of Definition \ref{first_representation},
define the exit time $Z^{\mathbf{0}\rightarrow
v_N}$ to be the non-zero coordinate of the point at which the stationary geodesic from $\mathbf{0}$ to $v_N$
exits the coordinate axes, the convention being that the exit time is
positive if the exit occurs along the $x$ axis and negative if it occurs
along the $y$ axis. We prove the following estimate for the right tail of
the exit time:

\begin{theorem}
	\label{exit_time}
	There exist positive constants $N_0$, $C,c$ that depend only on $\rho \in
	(0,1)$ such
	that for all $r>0$, $N\geq N_0$, we have
	\begin{displaymath}
		\mathbb{P}^{\rho}\left( |Z^{\mathbf{0}\rightarrow v_N}| \geq rN^{2/3}
		\right)\leq Ce^{-cr^3}.
	\end{displaymath}
\end{theorem}

As we already mentioned in the introduction, the above result was very
recently obtained in the paper \cite{EJS20}. There, it was obtained by using
explicit calculations for the log-moment generating function of the
stationary passage time. The lower bound with the optimal exponent for the
exit time is also known \cite{Sep17, BS10}. Some estimates for the exit
time along similar
lines are also available in \cite{FO18+}. Now, we state the upper tail
estimates that we obtain for the stationary last passage time along the characteristic line.
	\begin{theorem}
			\label{tails_upper_upper}
				For each fixed $\delta_1\in(0,1)$, there exist constants $C,c$
				depending only on $\delta_1, \rho\in
				(0,1)$ such
			that for all $N\geq N_0$ and $y$ satisfying $\delta_1
					N^{2/3}>y>0$ , we have 
			\begin{enumerate}
				\item	$\mathbb{P}^{\rho}(G^1_{\mathrm{stat}}(v_N)-N\geq
					yN^{1/3})\leq Ce^{-cy^{3/2}}$.
				\item $\mathbb{P}^{\rho}(G^1_{\mathrm{stat}}(v_N)-N\geq
				yN^{1/3})\geq Ce^{-cy^{3/2}}$.
			\end{enumerate}
		\end{theorem}
We now state the lower tail estimates that we obtain for the stationary last passage time
along the characteristic line.
		\begin{theorem}
			\label{tails_lower_lower}
			For any  fixed $\delta_1\in (0,1)$, there exist constants $C,c$
			depending on $\rho, \delta_1$ such
			that for all $N\geq N_0$ and $y$ satisfying $\delta_1
					N^{2/3}>y>0$, we have
			\begin{enumerate}
				\item $\mathbb{P}^{\frac{1}{2}}(G^1_{\mathrm{stat}}(v_N)-N\leq
					-yN^{1/3})\geq Ce^{-cy^3}$. 
				\item $\mathbb{P}^{\rho}(G^1_{\mathrm{stat}}(v_N)-N\leq
					-yN^{1/3})\leq Ce^{-cy^3}$.
			\end{enumerate}	
		\end{theorem}

	Though we give matching upper and lower bounds only for
	$\rho=\frac{1}{2}$, the upper bound that we give for general $\rho$ is
	also optimal. Indeed, in the limit $N\rightarrow \infty$, the normalized
	passage time
	$\frac{G^1_{\mathrm{stat}}(v_N)-N}{N^{1/3}}$ is known to converge in
	distribution
	to the Baik-Rains distribution which is known to have the tail
	estimates \cite{BR00} that we obtain for the finite $N$ case. At this
	point we are unable to obtain the lower bound for general $\rho$; the
	reader can refer to
	Remark \ref{difficulties} for a discussion of the difficulties. 
	
	As we mentioned earlier, the optimal right tail estimate for
	the exit time, as well as the optimal upper bound estimate for the upper tail were recently obtained in the paper
	\cite{EJS20} by a different approach. It is plausible that the
	explicit moment generating function calculations used in \cite{EJS20} can
	also be
	used to obtain the results for the lower tail, but we think that it is of
	value to observe that these can also be obtained by comparison with the
 exponential LPP models with the step and flat initial conditions.

\section{Technical Ingredients}
For exponential LPP, we know that
$\frac{\mathbb{E}[G(\mathbf{0},\alpha(m,n))]}{\alpha}\rightarrow (\sqrt{m}+\sqrt{n})^2$ as
$\alpha\rightarrow \infty$ \cite{Ro81,Jo00}. To reduce clutter, define
$f((m,n))=(\sqrt{m}+\sqrt{n})^2$.  Similarly, define
$g(x)=\frac{x}{1-\rho}$ for $x\geq 0$ and $g(x)=-\frac{x}{\rho}$ for
$x<0$. Finally, define $h(x)=\frac{\rho x^2}{4(1-\rho)^3 }$ for $x\geq 0$ and
$h(x)=\frac{(1-\rho)x^2}{4\rho^3}$ for $x<0$.

The
following lemma roughly says that for $\mathbf{x}$ where $1<x\ll N$,
$\mathbb{E}^{\rho}[G^1_{\mathrm{stat}}(\mathbf{x})]+\mathbb{E}^{\rho}[G(
\mathbf{x}^{\uparrow},v_N)]\simeq
\mathbb{E}^{\rho}[G^1_{\mathrm{stat}}(v_N)]-\frac{\rho x^2}{4(1-\rho)^3N}$, where $\mathbf{x}^{\uparrow}=(x,1)$. Note that
$\mathbb{E}^{\rho}[G(
\mathbf{x}^{\uparrow},v_N)]=\mathbb{E}[G(
\mathbf{x}^{\uparrow},v_N)]$ because in the positive quadrant, the boundary
representation of stationary
LPP differs from exponential LPP only at
the boundaries. Also note that
$\mathbb{E}^{\rho}[G^1_{\mathrm{stat}}(\mathbf{x})]=\frac{x}{1-\rho}$ for
$x>0$ because it is a sum of
$x$ independent exp($1-\rho$) weights with a similar corresponding statement holding for
$x<0$.
\begin{lemma}
	\label{expecatation_x^2/n}
	For all $x$ with $-\rho^2 N<x<(1-\rho)^2 N$ , we have
	\begin{displaymath}
		g(x)+f(v_N-\mathbf{x}^{\uparrow})=N-\frac{h(x)}{N}-N\mathcal{O}\left(
		(\frac{x}{N})^3\right),
	\end{displaymath}
	 where the $\mathcal{O}\left( (\frac{x}{N})^3\right)$ is a term
	that is strictly
	positive for all $x$ in the given range.
\end{lemma}
\begin{proof}
	The proof follows by plugging in the expression for $f(v_N)$ and doing a
	Taylor expansion. Note that $f(v_N)=N$.
\end{proof}

The main idea used in the proof of Theorem \ref{exit_time} comes from the
above lemma. Lemma \ref{expecatation_x^2/n} roughly shows that
$\mathbb{E}^{\rho}\left[G^1_{\mathrm{stat}}(\mathbf{x})+G(\mathbf{x}^{\uparrow},v_N)\right]$
is about $N-C\frac{x^2}{N}$. Note that
$\mathbb{E}^{\rho}\left[G^1_{\mathrm{stat}}(\mathbf{x})+G(\mathbf{x}^{\uparrow},v_N)\right]$
is the expected weight of the best up-right path from $\mathbf{0}$ to
$v_N$ which exits at $\mathbf{x}$. On the other hand, by comparison
with the point-to-point estimates for exponential LPP, we already know
that $\mathbb{E}^{\rho}[G^1_{\mathrm{stat}}(v_N)]$ is at least $N$, and we also have
upper bound estimates for the lower tail. Due to the discrepancy between the
means and the good concentration estimates about
their respective means (Proposition \ref{point_to_point} and Proposition
\ref{point_to_line}), it is
unlikely that we have
$G^1_{\mathrm{stat}}(\mathbf{x})+G(\mathbf{x}^{\uparrow},v_N)\geq
G^1_{\mathrm{stat}}(v_N)$. When done formally, this gives us an upper bound of the
probability of the exit time being exactly $\mathbf{x}$. We will finally do
it for a
range of $x$ simultaneously which adds technicality, but the basic idea is
still the same. 

One ingredient that we will use is the following point-to-point moderate
deviation estimate for exponential LPP coming from \cite{LR10}:
\begin{proposition}
	\label{point_to_point}
	For each $\psi>1$, there exist $C,c>0$ depending on $\psi$ such that
	for all $m,n$ sufficiently large with $\psi^{-1}<\frac{m}{n}<\psi$ and
	all $y>0$, we have the following:
	\begin{enumerate}
		\item $\mathbb{P}(G(\mathbf{0},(m,n))-(\sqrt{m}+\sqrt{n})^2\geq yn^{1/3})\leq
			Ce^{-c\min\{y^{3/2},yn^{1/3}\}}$.
		\item $\mathbb{P}(G(\mathbf{0},(m,n))-(\sqrt{m}+\sqrt{n})^2\leq -yn^{1/3})\leq
			Ce^{-cy^3}$.
	\end{enumerate}
\end{proposition}

For convenience, we have taken the above specific version of the result from
Theorem 4.1 in \cite{BGZ19}. We will also need a lower bound estimate for the upper tail for
point-to-point exponential LPP. It is obtained from Theorem 4 in
\cite{LR10} in the same way as Proposition \ref{point_to_point} is obtained
from using
results from \cite{LR10} as described in \cite{BGZ19}. Note that Theorem 4 in
\cite{LR10} is stated for Hermite ensembles, but as mentioned in \cite{LR10}, the same
technique works for the Laguerre case.

\begin{proposition}
	\label{import_upper_tail_lower_bound}
		For each $\psi>1$, there exist $C,c>0$ depending on $\psi$ such that
	for all $m,n$ sufficiently large with $\psi^{-1}<\frac{m}{n}<\psi$ and
	all $y>0$, we have the following:
	\begin{displaymath}
		\mathbb{P}(G(\mathbf{0},(m,n))-(\sqrt{m}+\sqrt{n})^2\geq
		yn^{1/3})\geq
		Ce^{-cy^{3/2}}.
	\end{displaymath}
\end{proposition}
Apart from using the point-to-point moderate deviation estimate, we will
also be using the point-to-line estimate:

\begin{proposition}
	\label{point_to_line}
Fix a $\rho\in (0,1)$ and $\delta_2>0$. Consider a line segment $\mathbb{L}_m(N)$ on $\{y=0\}$ with midpoint
	$(m(1-\rho)^2 N,0)$ and length $2N^{2/3}$. For each $\psi\in (0,1)$, there exists $C,c>0$
	(depending only on $\rho,\psi,\delta_2$) such that for all $|m|<\psi N^{1/3}$ and
	$y$ satisfying
	$\delta_2 N^{2/3}>y>0$, we have
\begin{displaymath}
	 \mathbb{P}\left(\max_{x\in
	 \mathbb{L}_m(n)}\left\{G(\mathbf{x}^{\uparrow},v_N)-\mathbb{E}[G(\mathbf{x}^{\uparrow},v_N)]\right\}>y
	 N^{1/3}\right) \leq Ce^{-cy^{3/2}}.
\end{displaymath}
\end{proposition}

Proposition \ref{point_to_line} is a special case of Theorem 10.5 in
\cite{BSS14} where it is written for the more general case of parallelograms.
Note that Theorem 10.5 in \cite{BSS14} gives an upper bound of $Ce^{-cy}$
but an inspection of the proof reveals that the exponent $y$ comes from
using a suboptimal point-to-point upper bound, but using the same argument
with the optimal point-to-point upper bound as in Proposition
\ref{point_to_point} gives the correct exponent of
$y^{3/2}$. Also, for Theorem 10.5 in \cite{BSS14}, it a-priori appears that
the vertices of the parallelogram are placed in a manner not resembling our
setting, but an inspection of the proof
shows that the vertices can be situated on any lines as long as the slopes
of the parallelogram edges of linear length are bounded away from the
coordinate directions. 

In the last section, we will need the following lower bound of the lower tail
probability of the point-to-line passage time in exponential LPP:
\begin{proposition}
	\label{lower_tail_lower_bound_ingredient}
	For any constant $\delta_2\in (0,4)$, there exist constants $c>0$,
	$n_0\in \mathbb{N}$ depending on $\delta_2$ such that for
all $n>n_0$ and $y\in(1,\delta_2 n^{2/3})$, we have
	\begin{displaymath}
		\mathbb{P}\left(\max_{t\in\mathbb{Z}}\left\{ G\left( (t,-t),(n,n) \right)
		\right\}\leq 4n-
	yn^{1/3}\right)\geq e^{-cy^3}.
	\end{displaymath}
\end{proposition}
The above proposition comes from Theorem 1.2 from \cite{BGHK19}. In \cite{BGHK19},
the result is stated for $\delta_2=1$, but any $\delta_2\in(0,4)$ works by
Theorem 2 along with the remarks at the end of the first section therein.

\section{Upper bound for the exit time}
We now proceed with the proof of Theorem \ref{exit_time}. To begin, we bound
the probability of the exit time lying in the interval
$[rN^{2/3},(r+1)N^{2/3}]$ where $r$ is an integer and $-\gamma_2 N^{2/3}<r<\gamma_1 N^{2/3}$ for
some fixed $\gamma_1,\gamma_2$ satisfying
$0<\gamma_1<(1-\rho)^2$ and $0<\gamma_2<\rho^2$. We first aim to show the following intermediate result in the
proof of Theorem \ref{exit_time}:
\begin{proposition}
	\label{r,r+1}
	Fix positive constants $\gamma_1,\gamma_2$ such that $(1-\rho)^2
	>\gamma_1>0$ and $\rho^2>\gamma_2>0$.
	There exist constants $C,c,N_0$ depending on $\gamma_1,\gamma_2,\rho$ such that for
	all integers $r$ with $-\gamma_2 N^{1/3}<r<\gamma_1 N^{1/3}$ and $N\geq N_0$, we have
	\begin{displaymath}
			\mathbb{P}^{\rho} (Z^{\mathbf{0}\rightarrow v_N}\in
			[rN^{2/3},(r+1)N^{2/3}])\leq Ce^{-c|r|^3}.
	\end{displaymath}
\end{proposition}
We split the above probability into two parts by using a union bound as
follows:
\begin{align}
	\mathbb{P}^{\rho} (Z^{\mathbf{0}\rightarrow v_N}\in
	[rN^{2/3},(r+1)N^{2/3}])&=\mathbb{P}^{\rho}\left(\max_{x\in
	[rN^{2/3},(r+1)N^{2/3}]}\{G^1_{\mathrm{stat}}(\mathbf{x})+G(\mathbf{x}^{\uparrow},v_N)\}\geq G^1_{\mathrm{stat}}(v_N)\right)\nonumber\\
	&\leq \mathbb{P}^{\rho}(G^1_{\mathrm{stat}}(v_N)\leq\alpha)+\mathbb{P}^{\rho}\left( \max_{x\in
	[rN^{2/3},(r+1)N^{2/3}]}\{G^1_{\mathrm{stat}}(\mathbf{x})+G(\mathbf{x}^{\uparrow},v_N)\}\geq\alpha
	\right).\label{split_alpha}
\end{align}
The above works for any $\alpha$ but to get good estimates, we need to choose
$\alpha$ such that both the terms in the above expression are small. That is,
$\alpha$ should be far enough from the means of both $G^1_{\mathrm{stat}}(v_N)$ and $\max_{x\in
	[rN^{2/3},(r+1)N^{2/3}]}\{G^1_{\mathrm{stat}}(\mathbf{x})+G(\mathbf{x}^{\uparrow},v_N)\}$.
	Combining Lemma \ref{expecatation_x^2/n} with the above intuition, we
	will set
	$\alpha$ to be about $N-\frac{h(rN^{2/3})}{2N}$. Note that
	$\frac{x^2}{N}=r^2N^{1/3}$ if $x=rN^{2/3}$. We now bound each of the
	terms in $\eqref{split_alpha}$. To bound the first term, we use
	Proposition \ref{point_to_point}.
	\begin{lemma}
		\label{first_term}
		There exists $N_0$ such that for some positive constants $C,c$
		depending on $\rho$, and for all $N>N_0$ and all $r$, we have
		that

		\begin{displaymath}
			\mathbb{P}^{\rho}\left(G^1_{\mathrm{stat}}(v_N)\leq N-
			r^2N^{1/3}\right)\leq
		Ce^{-c|r|^6}.
		\end{displaymath}
	\end{lemma}
	\begin{proof}
		The proof for the same statement under $\mathbb{P}(\cdot)$ instead of
		$\mathbb{P}^{\rho}(\cdot)$ would be a direct application of
		Proposition \ref{point_to_point}. To do it for $\mathbb{P}^{\rho}(\cdot)$, note
		that we have 
		$\left\{G^1_{\mathrm{stat}}(v_N)\leq
		N-r^2 N^{1/3}\right\}\subseteq
		\left\{G( (1,1),v_N)\leq (N-1)-(r^2 N^{1/3}-1)\right\}$
		which gives that
		\begin{align*}
				\mathbb{P}^{\rho}\left(G^1_{\mathrm{stat}}(v_N)\leq N-r^2
				N^{1/3}\right)&\leq
				\mathbb{P}^{\rho}\left(G((1,1),v_N)\leq (N-1)-(r^2 N^{1/3}-1)\right)\\
				&=
				\mathbb{P}\left(G((1,1),v_N)\leq (N-1)-(r^2 N^{1/3}-1)\right)\\
				&\leq Ce^{-c|r|^6}.
		\end{align*}
		Here, we used Proposition \ref{point_to_point} in the last step.
	\end{proof}

	Note that Lemma \ref{expecatation_x^2/n} says that for $x>0$,
	$\frac{x}{1-\rho}+f(v_N-\mathbf{x}^{\uparrow})$ decreases as $x$ increases. To simplify the coming expressions, define
	$\underline{x}=rN^{2/3}$ and $\overline{x}=(r+1)N^{2/3}$ for $r\geq0$ and
	$\underline{x}=(r+1)N^{2/3}$ and $\overline{x}=rN^{2/3}$ for $r<0$. All the $\max$
	symbols from now till the end of the second section denote the maximum over
	the variable $x$ varying in the interval $[rN^{2/3},(r+1)N^{2/3}]$. For the second term in
	\eqref{split_alpha}, we have
	\begin{align}
		&\mathbb{P}^{\rho}\left(\max\{G^1_{\mathrm{stat}}(\mathbf{x})+G(\mathbf{x}^{\uparrow},v_N)\}\geq
		N-\frac{h(\underline{x})}{2N}\right)\nonumber\\
&\leq\mathbb{P}^{\rho}\left(\max\left\{\left(G^1_{\mathrm{stat}}(\mathbf{x})-g(x)\right)+\left(G(\mathbf{x}^{\uparrow},v_N)-f(v_N-\mathbf{x}^{\uparrow})\right)\right\}\geq
	N-\frac{h(\underline{x})}{2N}-g(\underline{x})-f(v_N-\underline{\mathbf{x}}^\uparrow)\right)\nonumber\\
	&=	\mathbb{P}^{\rho}\left(\max\left\{\left(G^1_{\mathrm{stat}}(\mathbf{x})-g(x)\right)+\left(G(\mathbf{x}^{\uparrow},v_N)-f(v_N-\mathbf{x}^{\uparrow})\right)\right\}\geq \frac{h(\underline{x})}{2N}+N\mathcal{O}\left(
	(\frac{\underline{x}}{N})^3\right)\right)\nonumber\\ 
	&\leq	\mathbb{P}^{\rho}\left(\max\left\{\left(G^1_{\mathrm{stat}}(\mathbf{x})-g(x)\right)+\left(G(\mathbf{x}^{\uparrow},v_N)-f(v_N-\mathbf{x}^{\uparrow})\right)\right\}\geq \frac{h(\underline{x})}{2N}\right)\nonumber\\  
	&\leq	\mathbb{P}^{\rho}\left(\max\left\{\left(G^1_{\mathrm{stat}}(\mathbf{x})-g(x)\right)\right\}\geq
\frac{h(\underline{x})}{4N}\right)+\mathbb{P}^{\rho}\left(\max\left\{\left((G(\mathbf{x}^{\uparrow},v_N)-f(v_N-\mathbf{x}^{\uparrow})\right)\right\}\geq\frac{h(\underline{x})}{4N}\right).
	\label{second_split}
	\end{align}
	Note that the second inequality follows because the $\mathcal{O}\left(
	(\frac{x}{N})^3\right)$ in Lemma
	\ref{expecatation_x^2/n} is strictly
	positive. We again bound each of the terms in \eqref{second_split}
	separately. The first term is handled
	in the following lemma:

	\begin{lemma}
		\label{first_term_again}
	For any fixed positive $\gamma_1,\gamma_2$ such that $\gamma_1<(1-\rho)^2$ and
	$\gamma_2<\rho^2$, there exist
	positive constants $N_0,C,c$
		depending on $\rho,\gamma_1,\gamma_2$ such that for all $N>N_0$ and
		all integers $r$
		such that
		$\gamma_1
		N^{1/3}>r>-\gamma_2 N^{1/3}$, we have
		\begin{displaymath}
			\mathbb{P}^{\rho}\left(\max_{x\in [rN^{2/3},(r+1)N^{2/3}]
			}\left\{G^1_{\mathrm{stat}}(\mathbf{x})-g(x)\right\}\geq
			\frac{h(\underline{x})}{4N}\right)\leq
			Ce^{-c|r|^3}.
		\end{displaymath}
	\end{lemma}
	\begin{proof}
	Note that $M_n=G(\mathbf{0},\mathbf{n})-\frac{n}{1-\rho}$ where
		$n\in(rN^{2/3},(r+1)N^{2/3})$ is a martingale. Also note that we do
		the proof for the case $r\geq 0$. For negative $r$, the proof
		is the same, except that the martingale
		$M_{-n}'=G(-\mathbf{n},\mathbf{0})$ is used instead of
		$M_n$. Coming back to the case $r\geq 0$, on using Doob's maximal
		inequality for $M_n$, we get that for any $\lambda>0$ and
		$r>\frac{1}{3}$,
		\begin{equation*}
   \mathbb{P}^{\rho}\left(\max\left\{G^1_{\mathrm{stat}}(\mathbf{x})-\frac{x}{1-\rho}\right\}\geq \frac{N^{1/3}r^2\rho}{16(1-\rho)^3}\right)\leq
			\frac{1}{\exp\left(\frac{\lambda
			r^2\rho}{16\sqrt{r+1}(1-\rho)^2}\right)}\mathbb{E}^{\rho}\left[
			\exp\left(\frac{\lambda
			(G^1_{\mathrm{stat}}(\mathbf{\overline{x}})-\frac{\overline{x}}{1-\rho})}{\frac{\sqrt{r+1}N^{1/3}}{1-\rho}}
			\right)\right].
		\end{equation*}
		Using that $r>\frac{1}{3}$, we finally get that
		\begin{equation}
			\label{doobs_maximal}
   \mathbb{P}^{\rho}\left(\max\left\{G^1_{\mathrm{stat}}(\mathbf{x})-\frac{x}{1-\rho}\right\}\geq \frac{N^{1/3}r^2\rho}{16(1-\rho)^3}\right)\leq \frac{1}{\exp\left(\frac{\lambda
			r^{3/2}\rho}{32(1-\rho)^2}\right)}\mathbb{E}^{\rho}\left[
			\exp\left(\frac{\lambda
			(G^1_{\mathrm{stat}}(\mathbf{\overline{x}})-\frac{\overline{x}}{1-\rho})}{\frac{\sqrt{r+1}N^{1/3}}{1-\rho}}
			\right)\right].
		\end{equation}
		The above term is bounded by using that $\mathbb{E}^{\rho}\left[
			\exp\left(\frac{\lambda
			(G^1_{\mathrm{stat}}(\mathbf{\overline{x}})-\frac{\overline{x}}{1-\rho})}{\frac{\sqrt{r+1}N^{1/3}}{1-\rho}}\right)\right]\leq
			e^{C^*\lambda^2}$ for a proper choice of parameters, and some
			constant $C^*$ depending on $\gamma_1,\rho$. This is
			formally done in Lemma \ref{mgf_converges1} which is proved
			using the technical Lemma \ref{mgf_converges}; the proofs are 
			routine and are
			moved to the appendix.	
			Returning to the proof of the lemma, looking at \eqref{doobs_maximal},
			we choose $\lambda$ so that it minimizes
			$C^*\lambda^2-\frac{\lambda r^{3/2}\rho}{32(1-\rho)^2}$,
			that is, we choose $\lambda=\frac{r^{3/2}\rho}{64C^*(1-\rho)^2}$,
			and the value of the above expression for this choice of
			$\lambda$ is $-\frac{r^3\rho^2}{2^{12}C^*(1-\rho)^4}$. Plugging in
			this value of $\lambda$ in \eqref{doobs_maximal}, we get that for
			$N$ large enough, and for $\gamma_1 N^{1/3} > r>
			\frac{1}{3}\lor r_0$,
			\begin{equation}
				\label{first_part_of_second_term}
					\mathbb{P}^{\rho}\left(\max\left\{G^1_{\mathrm{stat}}(\mathbf{x})-\frac{x}{1-\rho}\right\}\geq \frac{N^{1/3}r^2\rho}{16(1-\rho)^3}\right)\leq
					e^{-\frac{r^3\rho^2}{2^{12}C^*(1-\rho)^4}}.
			\end{equation}
			This is what we wanted to prove. To include all $0< r \leq
			r_0\lor \frac{1}{3}$,
			just adjust the values of the constants $C,c$ in the
			statement of the lemma.
		\end{proof}
	\begin{lemma}
				\label{mgf_converges}
				For any constant $C^*$ with $C^*>
				\frac{1}{2}$, there exists a constant $\delta_0 \in(0,1)$
				depending on $C^*,\rho$ such
				that for all $0<r<(1-\rho)^2N^{1/3}$, $N>0$ and $\lambda>0$ satisfying $0<\lambda<\delta_0
				\sqrt{\overline{x}}$, we
				have
				\begin{displaymath}
					\mathbb{E}^{\rho}\left[
			\exp\left(\frac{\lambda
			(G^1_{\mathrm{stat}}(\mathbf{\overline{x}})-\frac{\overline{x}}{1-\rho})}{\frac{\sqrt{r+1}N^{1/3}}{1-\rho}}\right)\right]\leq
			e^{C^*\lambda^2}.
				\end{displaymath}
				Note that $\delta_0$ can be chosen such
				that $\delta_0\rightarrow 1$ as $C^*\rightarrow
				\infty$.
			\end{lemma}

		\begin{lemma}
				\label{mgf_converges1}
			For any positive constant $\gamma_1$ such that $\gamma_1<(1-\rho)^2$,
			there exist positive constants 
				$C^*,N_1,r_0$ depending on $\rho,\gamma_1$ such that for
				$r_0<r<\gamma_1 N^{1/3}$ and
				$\lambda=\frac{r^{3/2}\rho}{64C^*(1-\rho)^2}$ and $N\geq N_1$,
				we have
				\begin{displaymath}
					\mathbb{E}^{\rho}\left[
			\exp\left(\frac{\lambda
			(G^1_{\mathrm{stat}}(\mathbf{\overline{x}})-\frac{\overline{x}}{1-\rho})}{\frac{\sqrt{r+1}N^{1/3}}{1-\rho}}\right)\right]\leq
			e^{C^*\lambda^2}.
				\end{displaymath}
			\end{lemma}
			
			As we mentioned earlier, the proofs of Lemma \ref{mgf_converges}
			and Lemma \ref{mgf_converges1} have been postponed to the
			appendix. We now bound the second term in \eqref{second_split}, that is, we
			aim to show the following lemma:
			\begin{lemma}
				\label{second_split_second_term_lemma}
			For any fixed positive $\gamma_1,\gamma_2$ with $\gamma_1<(1-\rho)^2$
			and $\gamma_2<\rho^2$, there exist 
				$N_0,C,c$
				depending on $\rho,\gamma_1,\gamma_2$ such that for all $N>N_0$ and all
				integers $r$ with
				$\gamma_1 N^{1/3}>r>-\gamma_2 N^{1/3}$, we have that
				\begin{displaymath}
					\mathbb{P}^{\rho}\left(\max_{x\in
					[rN^{2/3},(r+1)N^{2/3}]}\left\{G(\mathbf{x}^{\uparrow},v_N)-f(v_N-\mathbf{x}^{\uparrow})\right\}\geq
					\frac{h(\underline{x})}{4N}\right)\leq
					Ce^{-c |r|^3}.
				\end{displaymath}
			\end{lemma}
			\begin{proof}
			To begin, note that we only need to show the above result for
			$|r|>r_0$ for some positive constant $r_0$ since
			we can handle the case of small $r$ by adjusting the constants
			$C,c$. We only do it for the case $r>0$; the other case is
			analogous. Now, for a specific $C'>0$, we have that
			\begin{align}
	   &\mathbb{P}^{\rho}\left(\max\left\{G(\mathbf{x}^{\uparrow},v_N)-f(v_N-\mathbf{x}^{\uparrow})\right\}\geq \frac{N^{1/3}r^2\rho}{16(1-\rho)^3}\right)\nonumber\\
		  &\leq
		  \mathbb{P}^{\rho}\left(\max\left\{G(\mathbf{x}^{\uparrow},v_N)-\mathbb{E}^{\rho}[G(\mathbf{x}^{\uparrow},v_N)]\right\}\geq \frac{N^{1/3}r^2\rho}{16(1-\rho)^3}
		  -C'N^{1/3}\right). \label{second_split_second_term}
			\end{align}
			To get the above expression, the constant $C'$ is chosen
			so that
			$|\mathbb{E}^{\rho}[G(\mathbf{x}^{\uparrow},v_N)]-f(v_N-\mathbf{x}^{\uparrow})|<
			C' N^{1/3}$ for all $x\in[rN^{2/3},(r+1)N^{2/3}]$. Indeed, for a
			fixed choice of $\gamma_1$ and $\gamma_2$, $-\gamma_2 N^{2/3}<r<\gamma_1 N^{2/3}$ implies that for all $z\in
			(-\gamma_2 N^{2/3},\gamma_1 N^{2/3})$, we have that the straight line joining
			$\mathbf{z}^{\uparrow}$ and $v_N$ has slope uniformly bounded away from $0$
			and $\infty$. Hence, we are in the setting of Proposition
			\ref{point_to_point} which implies that for some constant $C'$
			(depending only on $\gamma_1,\gamma_2,\rho$), we have
			$|\mathbb{E}^{\rho}[G(\mathbf{x}^{\uparrow},v_N)]-f(v_N-\mathbf{x}^{\uparrow})|<
			C' N^{1/3}$ for all $x\in[rN^{2/3},(r+1)N^{2/3}]$.
			Choosing $r_0$ large depending on $C'$, we get that 
			\begin{align}
		&\mathbb{P}^{\rho}\left(\max\left\{G(\mathbf{x}^{\uparrow},v_N)-\mathbb{E}^{\rho}[G(\mathbf{x}^{\uparrow},v_N)]\right\}\geq\frac{N^{1/3}r^2\rho}{16(1-\rho)^3}
		  -C'N^{1/3}\right)\nonumber \\
		  &\leq
		  \mathbb{P}^{\rho}\left(\max\left\{G(\mathbf{x}^{\uparrow},v_N)-\mathbb{E}^{\rho}[G(\mathbf{x}^{\uparrow},v_N)]\right\}\geq \frac{N^{1/3}r^2\rho}{32(1-\rho)^3}\right)
		  \label{second_split_first_term}.
			\end{align}
			To finish, just observe that the final expression fits exactly in
			the setting of Proposition \ref{point_to_line}, and note that
			$(r^2)^{3/2}=r^3$.
			\end{proof}

			\begin{proof}[Proof of Proposition \ref{r,r+1}]
				By \eqref{split_alpha} and \eqref{second_split}, we have
				\begin{align*}
					&\mathbb{P}^{\rho} (Z^{\mathbf{0}\rightarrow v_N}\in
	[rN^{2/3},(r+1)N^{2/3}])\\
	&\leq \mathbb{P}^{\rho}\left(G^1_{\mathrm{stat}}(v_N)\leq N-
	\frac{h(\underline{x})}{2N}\right)+\mathbb{P}^{\rho}\left( \max_{x\in
	[rN^{2/3},(r+1)N^{2/3}]}\{G^1_{\mathrm{stat}}(\mathbf{x})+G(\mathbf{x}^{\uparrow},v_N)\}\geq \frac{h(\underline{x})}{2N}
	\right)\\
	&\leq C_1e^{-c_1 r^6}
	+\mathbb{P}^{\rho}\left(\max\left\{G^1_{\mathrm{stat}}(\mathbf{x})-g(x)\right\}\geq
	\frac{h(\underline{x})}{4N}\right)+\mathbb{P}^{\rho}\left(\max\left\{G(\mathbf{x}^{\uparrow},v_N)-f(v_N-\mathbf{x}^{\uparrow})\right\}\geq
	\frac{h(\underline{x})}{4N}\right)\\
	&\leq C_1e^{-c_1r^6}+C_2e^{-c_2|r|^3}+C_3e^{-c_3|r|^3}\leq Ce^{-c|r|^3}.
				\end{align*}
		The first term in the fourth line was obtained by using Lemma
		\ref{first_term} and the next two terms in the fourth line were
		obtained by using Lemma \ref{first_term_again} and
		Lemma \ref{second_split_second_term_lemma}.
			\end{proof}
		 We now use Proposition \ref{r,r+1} to prove Theorem \ref{exit_time}.

			\begin{proof}[Proof of Theorem \ref{exit_time}]
				We will show that for all
				$r'>0$ and $N$ sufficiently large,
				\begin{displaymath}
					\mathbb{P}^{\rho}\left( Z^{\mathbf{0}\rightarrow v_N} \geq r'N^{2/3}
		\right)\leq Ce^{-cr'^3}.
				\end{displaymath}
				Note that we are only doing the proof for $r'>0$ but the
				proof for $r'<0$ is the same with the role of $\gamma_1$
				being replaced by $\gamma_2$. Note that we only need to worry about $r'>r_0$ as we can
				adjust the constants to get the result for small $r'$. For any
				fixed positive
				$\gamma_1$ with $\gamma_1<(1-\rho)^2$, note that
				\begin{equation}
					\label{main_from_parts}
					\mathbb{P}^{\rho}\left( Z^{\mathbf{0}\rightarrow v_N} \geq r'N^{2/3}
					\right)\leq \sum_{i=0}^{\gamma_1
					N^{1/3}-r'} \mathbb{P}^{\rho}\left( Z^{\mathbf{0}\rightarrow v_N} \in
		\left[(r'+i)N^{2/3},(r'+i+1)N^{2/3}\right]
		\right) + \mathbb{P}^{\rho}\left( Z^{\mathbf{0}\rightarrow v_N}\geq
		\gamma_1 N
		\right).
				\end{equation}
				Note that the first term is present only if $r'\leq \gamma_1
				N^{1/3}$. The first term involving the sum can now be bounded by using Proposition
 \ref{r,r+1} as follows--
 \begin{displaymath}
	 \sum_{i=0}^{\gamma_1
					N^{1/3}-r'} \mathbb{P}^{\rho}\left( Z^{\mathbf{0}\rightarrow v_N} \in
		\left[(r'+i)N^{2/3},(r'+i+1)N^{2/3}\right]
		\right)\leq \sum_{i=0}^{\infty}C_1e^{-c_1 (r'+i)^3}\leq
		C_1'e^{-c_1'r'^3}.
 \end{displaymath}
 Here $C_1',c_1'$ are constants depending on $\gamma_1,\rho$. We will be
 choosing a specific value of $\gamma_1$ later in the argument. It now
 remains to bound the second term in \eqref{main_from_parts}. It turns out
 that the far end of the tail is easy to bound by a different argument. Note that we have the following
	  crude estimate:
	  \begin{align}
		  \label{crude}
		  \mathbb{P}^{\rho}\left(
		  Z^{\mathbf{0}\rightarrow v_N}\geq \gamma_1 N \right)\leq
		  \mathbb{P}^{\rho}\left(G(\mathbf{0},( (1-\rho)^2 N,0))+G(
		  (\gamma_1
		  N,1),v_N)>G^1_{\mathrm{stat}}(v_N)\right).
	  \end{align}
	 The same strategy used in the first term above works again if we can
	  choose $\gamma_1$ so that
	  $\mathbb{E}^{\rho}[G^1_{\mathrm{stat}}(( (1-\rho)^2 N,0))]+f(v_N-(\gamma_1 N,1))$ is at most
	  $N-\beta N$ for some $0<\beta<1$ depending on $\rho$. Noting that
	  $\mathbb{E}^{\rho}[G^1_{\mathrm{stat}}(( (1-\rho)^2 N,0))]=(1-\rho)N$, we have 
	  \begin{displaymath}
		  \mathbb{E}^{\rho}[G^1_{\mathrm{stat}}(( (1-\rho)^2 N,0))]+f(v_N-(\gamma_1 N,1))=
		  N\left(
		  (1-\rho)+\left( \sqrt{\rho^2-\frac{1}{N^2}}+(1-\rho)\sqrt{1-\frac{\gamma_1}{(1-\rho)^2}}\right)^2\right).
	  \end{displaymath}
	  At $\gamma_1=(1-\rho)^2$, the coefficient of $N$ is the above expression
	  is at most $(1-\rho)+\rho^2<(1-\rho)+\rho=1$. Hence, by continuity, we can choose
	  $\gamma_1$ sufficiently close to $(1-\rho)^2$ and obtain a positive
	  value of $\beta$ as needed. Thus, we now have for all $N$ large
	  enough,
	  \begin{align}
		  \mathbb{E}^{\rho}[G^1_{\mathrm{stat}}(( (1-\rho)^2 N,0))]+f(v_N-(\gamma_1 N,1))\leq
		  N-\beta N. \label{super_mean_small}
	  \end{align}
	  Hence, by using \eqref{super_mean_small} along with \eqref{crude}, we finally have that
	  \begin{align}
		   &\mathbb{P}^{\rho}\left(
		  Z^{\mathbf{0}\rightarrow v_N}\geq \gamma_1 N \right)\nonumber\\
		  &\leq \mathbb{P}^{\rho}\left(
		  G^1_{\mathrm{stat}}(v_N)<N-\frac{\beta N}{2}
		  \right)+\mathbb{P}^{\rho}\left(G^1_{\mathrm{stat}}(( (1-\rho)^2
		  N,0))-(1-\rho)N>\frac{\beta N}{4} \right)\nonumber\\
		  &+
		  \mathbb{P}^{\rho}\left( G( (\gamma_1
		  N,1),v_N)-f(v_N-(\gamma_1 N,1))>
		  \frac{\beta N}{4} \right). \label{a_lot_split}
	  \end{align}
 We can now repeat the arguments in Lemma
 \ref{first_term} and Lemma
			\ref{second_split_second_term_lemma} to bound the first and
			third terms. Indeed, the arguments are only made easier because there is
			no $\max$ involved. Note that the point-to-line estimate--
			Proposition \ref{point_to_line} used in
			Lemma \ref{second_split_second_term_lemma} is now substituted
			with the point-to-point estimate-- Proposition
			\ref{point_to_point}. The second term is bounded by using
			exponential concentration for sums of i.i.d.\ random variables. Hence, for large enough $N$, we have
			\begin{displaymath}
			 \mathbb{P}^{\rho}\left(
			 Z^{\mathbf{0}\rightarrow v_N}\geq \gamma_1 N \right)\leq C_4 e^{-c_4 N}.	
			\end{displaymath}
			To finish the proof, we go back to \eqref{main_from_parts}. For
			the case $r'\leq\gamma_1
			N^{1/3}$, we have
			\begin{displaymath}
					\mathbb{P}^{\rho}\left( Z^{\mathbf{0}\rightarrow v_N} \geq r'N^{2/3}
					\right) \leq C_1'e^{-c_1'r'^3`}+ C_4 e^{-c_4N}\leq
					C_5e^{-c_5r'^3}.
			\end{displaymath}
			For the case $r'>\gamma_1 N^{1/3}$, from
			\eqref{main_from_parts}, we have that
			\begin{displaymath}
					\mathbb{P}^{\rho}\left( Z^{\mathbf{0}\rightarrow v_N} \geq r'N^{2/3}
					\right) \leq C_4e^{-c_4 N}\leq Ce^{-cr'^3}.
			\end{displaymath}
			Note that the last inequality follows because we can restrict to
			$r'\leq (1-\rho)^2 N^{1/3}$ because Theorem \ref{exit_time} is
			vacuously true for the case $r'>(1-\rho)^2 N^{1/3}$. Indeed,
			$Z^{\mathbf{0}\rightarrow v_N}$ is deterministically smaller than
			$(1-\rho)^2 N$.
	\end{proof}

	\section{Upper tail estimates}

		The proof is along the same lines as the proof of Theorem
			\ref{exit_time} in the previous section; indeed, some
			ingredients are
			already proven implicitly in the last section. The proof of  (2)
			in Theorem \ref{tails_upper_upper} is straightforward by
			comparison to the point-to-point LPP estimates and we prove it
			now.

			\begin{proof}[Proof of (2) in Theorem \ref{tails_upper_upper}]
				Note that
				\begin{align*}
					\mathbb{P}^{\rho}(G^1_{\mathrm{stat}}(v_N)-N\geq
					yN^{1/3})&\geq 	\mathbb{P}(G( (1,1),v_N)-N\geq
					yN^{1/3})\nonumber\\
					&= \mathbb{P}(G( (1,1),v_N)-(N-1)\geq
					yN^{1/3}+1).
				\end{align*}
			Now, note that we are in the setting of Proposition
			\ref{import_upper_tail_lower_bound} which gives the result
			immediately.
			\end{proof}

			We now give a series of intermediate results with the aim of
			proving Theorem \ref{tails_upper_upper}. 
			\begin{proposition}
				\label{upper_upper}
			For each fixed $\delta_1\in (0,1)$ and  $\gamma_1,\gamma_2$ such that
			$(1-\rho)^2>\gamma_1>0$ and $\rho^2>\gamma_2>0$, there exist positive constants 
				$C,c,N_0$ depending on $\rho,\gamma_1,\gamma_2,\delta_1$ such
				that for all
				$N\geq N_0$ and $y$ satisfying $0<y<\delta_1 N^{2/3}$, we have
				\begin{enumerate}
					\item $\mathbb{P}^{\rho}\left( \max_{x\in
				[1,\gamma_1
				N]}\{G^1_{\mathrm{stat}}(\mathbf{x})+G(\mathbf{x}^{\uparrow},v_N)\}>N+yN^{1/3}\right)\leq
				Ce^{-cy^{3/2}}$.
			\item $\mathbb{P}^{\rho}\left( \max_{x\in
				[-\gamma_2 N,-1]}\{G^1_{\mathrm{stat}}(\mathbf{x})+G(\mathbf{x}^{\uparrow},v_N)\}>N+yN^{1/3}\right)\leq
				Ce^{-cy^{3/2}}$.
				\end{enumerate}
			\end{proposition}
			\begin{proof}
				Note that we can prove the proposition for all $y> y_0$
				for some $y_0$ depending only on $\rho$ and then adjust
				$C$ and $c$ to account for $y\leq y_0$. We will only prove the
				first part of the proposition; the proof of the second part
				is analogous. As we did in the first section, we first control the maximum
				in an interval of width $N^{2/3}$.
				Following \eqref{second_split}, for any $0< r<\gamma_1
				N^{1/3}$, where $\gamma_1,\gamma_2$ will be fixed later, write
				\begin{align}
				&\mathbb{P}^{\rho}\left( \max_{x\in
				[rN^{2/3},(r+1)N^{2/3}]}\{G^1_{\mathrm{stat}}(\mathbf{x})+G(\mathbf{x}^{\uparrow},v_N)\}>N+yN^{1/3}\right)\nonumber\\
				&\leq
   \mathbb{P}^{\rho}\left(\max\left\{G^1_{\mathrm{stat}}(\mathbf{x})-\frac{x}{1-\rho}\right\}>\frac{yN^{1/3}}{2}+\frac{N^{1/3}r^2\rho}{16(1-\rho)^3}\right)\nonumber\\
		 &+\mathbb{P}^{\rho}\left(\max\left\{G(\mathbf{x}^{\uparrow},v_N)-f(v_N-\mathbf{x}^{\uparrow})\right\}>\frac{yN^{1/3}}{2}+\frac{N^{1/3}r^2\rho}{16(1-\rho)^3}\right).\label{upper_tail_upper_bound}
				\end{align}
				  On repeating the proof of Lemma
				 \ref{second_split_second_term_lemma}, we get that for
				 all $0< r <\gamma_1 N^{2/3}$, we have
				 \begin{equation}
					 \label{first}
					 \mathbb{P}^{\rho}\left(\max\left\{G(\mathbf{x}^{\uparrow},v_N)-f(v_N-\mathbf{x}^{\uparrow})\right\}>\frac{yN^{1/3}}{2}+\frac{N^{1/3}r^2\rho}{16(1-\rho)^3}\right)\leq
					 C_1e^{-\left(\frac{y}{2}+\frac{r^2\rho}{16(1-\rho)^3}-C'\right)^{3/2}}.
				 \end{equation}
				 In the above expression $C'$ comes from the proof of
				 Lemma \ref{second_split_second_term_lemma}. This handles the second term. We will now
				 bound the first term for all $y$ such that
				 $y< \delta_1 N^{2/3}$ by using the technique in Lemma
				 \ref{first_term_again}. We have
				 \begin{equation}
					 \label{choose_lambda}
   \mathbb{P}^{\rho}\left(\max\left\{G^1_{\mathrm{stat}}(\mathbf{x})-\frac{x}{1-\rho}\right\}>\frac{yN^{1/3}}{2}+\frac{N^{1/3}r^2\rho}{16(1-\rho)^3}\right)\leq
 \frac{1}{\exp\left(\frac{\lambda
					 r^2\rho}{16\sqrt{r+1}(1-\rho)^2}+ \frac{\lambda y(1-\rho)}{2\sqrt{r+1}}\right)}\mathbb{E}^{\rho}\left[
			\exp\left(\frac{\lambda
			(G^1_{\mathrm{stat}}(\mathbf{\overline{x}})-\frac{\overline{x}}{1-\rho})}{\frac{\sqrt{r+1}N^{1/3}}{1-\rho}}
			\right)\right].
				 \end{equation}
				Since we have assumed that $y< \delta_1 N^{2/3}$,
				we have that
				$\sqrt{r+1}\sqrt{y}<\delta_1\sqrt{\overline{x}}$. Hence, by
				Lemma \ref{mgf_converges}, with $C^*$ such that
				$\delta_0(C^*)\geq \delta_1$, we can choose
				$\lambda=\sqrt{r+1}\sqrt{y}$ to finally get that for all
				$0<r<\gamma_1 N^{2/3}$,					 
				\begin{align}
						 \label{random_walk_upper_upper}
   \mathbb{P}^{\rho}\left(\max\left\{G^1_{\mathrm{stat}}(\mathbf{x})-\frac{x}{1-\rho}\right\}>\frac{yN^{1/3}}{2}+\frac{N^{1/3}r^2\rho}{16(1-\rho)^3}\right)
						 &\leq \exp\left\{C^* (r+1)y -
						 \left(\frac{\sqrt{y}r^{2}\rho}{16(1-\rho)^2}+\frac{y^{3/2}(1-\rho)}{2}\right)
						 \right\}\nonumber\\ 
						 &=\exp\left\{C^*\left(  (r+1)y -
						 c_3\sqrt{y}r^2-c_4
						 y^{3/2}\right)\right\}\nonumber\\
						 &=\exp\left\{C^*\left( -c_3\sqrt{y}\left(
						 (r-\frac{\sqrt{y}}{2c_3})^2-\frac{y}{4c_3^2}-\frac{\sqrt{y}}{c_3}
						 \right) -c_4y^{3/2} \right)\right\}.
					 \end{align}	
					 We will be using the above expression for the case
					 $r<\alpha \sqrt{y}$, where $\alpha$ will be chosen
					 later. For
					 the case $r\geq\alpha\sqrt{y}$, we choose a
					 different $\lambda$ in \eqref{choose_lambda}.  Note that
					 from proof of Lemma
					 \ref{mgf_converges1},
					 $\frac{r^2}{32C^{**}\sqrt{r+1}(1-\rho)^2}\times
				 \frac{1}{\sqrt{\overline{x}}}$ is smaller than
				 $\delta_0(C^{**})$ (where $\delta_0$ comes from Lemma
				 \ref{mgf_converges}) if $y_0,N_0,C^{**}$ are chosen properly. Hence, we
				 can choose $\lambda=\frac{r^2}{32C^{**}\sqrt{r+1}(1-\rho)^2}$, and
				this gives that for $r\geq\alpha \sqrt{y}$, (increasing
				$y_0$ if necessary)
				\begin{equation}
					\label{}
   \mathbb{P}^{\rho}\left(\max\left\{G^1_{\mathrm{stat}}(\mathbf{x})-\frac{x}{1-\rho}\right\}>\frac{yN^{1/3}}{2}+\frac{N^{1/3}r^2\rho}{32(1-\rho)^3}\right)\leq
					\exp\left\{ -\frac{r^3\rho}{2^{12}C^{**}}(1-\rho)^2
					-\frac{ry}{2^7(1-\rho)C^{**}}\right\}.
				\end{equation}
				Finally, we piece this together to prove Proposition
				\ref{upper_upper}. Observe that
				\begin{align*}
					\label{big_sum}
					&\mathbb{P}^{\rho}\left( \max_{x\in
				[1,\gamma_1
				N]}\{G^1_{\mathrm{stat}}(\mathbf{x})+G(\mathbf{x}^{\uparrow},v_N)\}>N+yN^{1/3}\right)\\
				&\leq
				\sum_{r=0}^{\gamma_1
				N^{1/3}-1}\mathbb{P}^{\rho}\left( \max_{x\in
				[rN^{2/3},(r+1)N^{2/3}]}\{G^1_{\mathrm{stat}}(\mathbf{x})+G(\mathbf{x}^{\uparrow},v_N)\}>N+yN^{1/3}\right).
				\end{align*}
				Note that each term in the above equation was split into two
				terms. The first term was bounded in \eqref{first}, and note
				that by possibly increasing $y_0$, we have
				\begin{displaymath}
					\sum_{r=0}^\infty
					C_1e^{-\left(\frac{y}{2}+\frac{r^2\rho}{16(1-\rho)^3}-C'\right)^{3/2}}\leq
					C_1'e^{-c_1'y^{3/2}}.
				\end{displaymath}
				This handles the contribution to the sum coming from the
				first term. For the second term, we obtained two different bounds
				depending on the value of $r$. Choose $\alpha>0$ such that
				$0<\alpha<\frac{1}{2c_3}$ and
				\begin{displaymath}
					-c_3\left( (\alpha-\frac{1}{2c_3})^2
					-\frac{1}{4c_3^2} \right)-c_4<0.
				\end{displaymath}
				Now,
				note that
				\begin{displaymath}
					\sum_{r=0}^{\alpha \sqrt{y}}
					\exp\left\{C^*\left( -c_3\sqrt{y}\left(
						 (r-\frac{\sqrt{y}}{2c_3})^2-\frac{y}{4c_3^2}-\frac{\sqrt{y}}{c_3}
						 \right) -c_4y^{3/2}\right) \right\}\leq
						 C_2'e^{-c_2'y^{3/2}}.
				\end{displaymath}
				Finally, also note that
				\begin{displaymath}
					\sum_{r=\alpha
					\sqrt{y}}^{\infty}\exp\left\{
					-\frac{r^3\rho}{2^{12}C^{**}}(1-\rho)^2
					-\frac{ry}{2^7(1-\rho)C^{**}}\right\}\leq
					C_3'e^{-c_3'y^{3/2}}.				
				\end{displaymath}
			Thus we finally have that 
			\begin{displaymath}
					\mathbb{P}^{\rho}\left( \max_{x\in
				[1,\gamma_1
				N]}\{G^1_{\mathrm{stat}}(\mathbf{x})+G(\mathbf{x}^{\uparrow},v_N)\}>N+yN^{1/3}\right)\leq
				C_1'e^{-c_1'y^{3/2}}+C_2'e^{-c_2'y^{3/2}}+C_3'e^{-c_3'y^{3/2}}\leq
				Ce^{-cy^{3/2}}.
			\end{displaymath}
			This finishes the proof of the proposition.
			\end{proof}
			\begin{proof}[Proof of (1) in Theorem \ref{tails_upper_upper}]
				Fix any positive $\gamma_1,\gamma_2$ such that
				$\gamma_1<(1-\rho)^2$ and $\gamma_2<\rho^2$. Observe that 
				\begin{align*}
					\mathbb{P}^{\rho}(G^1_{\mathrm{stat}}(v_N)-N\geq
					yN^{1/3})
					&\leq \mathbb{P}^{\rho}\left( \max_{x\in
				[1,\gamma_1
				N]}\{G^1_{\mathrm{stat}}(\mathbf{x})+G(\mathbf{x}^{\uparrow},v_N)\}>N+yN^{1/3}\right)\\&+ \mathbb{P}^{\rho}\left( \max_{x\in
				[-\gamma_2 N,-1]}\{G^1_{\mathrm{stat}}(\mathbf{x})+G(\mathbf{x}^{\uparrow},v_N)\}>N+yN^{1/3}\right)+
				\mathbb{P}^{\rho}\left( |Z^{\mathbf{0}\rightarrow v_N}| \geq
				(\gamma_1\wedge \gamma_2) N
		\right)\\
		&\leq C_4'e^{-c_4'y^{3/2}}+ C_5'e^{-c_5'y^{3/2}}+C_6'e^{-c_6'N^3}\\
		&\leq Ce^{-c y^{3/2}}.
				\end{align*}
				The first inequality is a straightforward union bound. To get
				the second inequality, we have used Proposition
				\ref{upper_upper} and Theorem \ref{exit_time} respectively. To
				get the final inequality, we have used that $y\leq \delta_1
				N^{2/3}$.
			\end{proof}

			\section{Lower tail estimates}
		 The proof of (2) in Theorem \ref{tails_lower_lower} is
		straightforward by comparison to the point-to-point estimates.
		\begin{proof}[Proof of (2) in Theorem \ref{tails_lower_lower}]
				The proof is exactly the same as the proof of Lemma
				\ref{first_term}. One just needs to replace
				$r^2$ in the statement of Lemma
				\ref{first_term} by $y$ and then reproduce the
				proof verbatim.
			\end{proof}
			We now come to the proof of (1) in Theorem
			\ref{tails_lower_lower}. For the remainder of the section, we
			will be working with $\rho=\frac{1}{2}$.
			We will now use the point-to-line representation of the stationary model. This helps us set up for a direct application of
			Proposition \ref{lower_tail_lower_bound_ingredient} which gives us
			the required lower bound. As we mentioned earlier, the two
			representations of the stationary model are equivalent, and
			this can be shown by coupling both the representations to
			a separate boundary representation model. The coupling
			is described in detail in the appendix, and the following
			proposition follows directly from the coupling:
			\begin{proposition}
				\label{couple}
				For $v\in \mathbb{Z}^2_{\geq 0}$, the random variables
				$\{G^1_{\mathrm{stat}}(v)\}$ have the same joint
				distribution as the random variables
				$\{G^2_{\mathrm{stat}}(v)\}$.
			\end{proposition}
			Recall the notation $\overline{\mathbb{P}}^{\rho}(\cdot)$ for the
			probabilities in the point-to-line stationary model. To be able to reduce a calculation in the point-to-line
			representation to one in the boundary representation, we will
			need the following result which is proved by the coupling
			argument in the appendix:
			\begin{proposition}
				\label{geodesic_coupling}
			 	In the point-to-line stationary representation, consider the
				non zero coordinate $q_N$ of the point where the stationary geodesic to $v_N$ last meets the coordinate
				axes, the convention being that $q_N$ is positive if the point is on
				the $x$ axis and is negative otherwise. Then we have that the
				distribution of $q_N$ is the same as the
				distribution of the exit time in the boundary
				representation. That is,
				\begin{displaymath}
					q_N \stackrel{\mathcal{D}}{=} Z^{\mathbf{0}\rightarrow
					v_N},
				\end{displaymath}
				where the distribution of the former is considered under
				$\overline{\mathbb{P}}^{\rho}(\cdot)$ and the latter under
				$\mathbb{P}^{\rho}(\cdot)$.
			\end{proposition}
			 Coming back to (1) in Theorem \ref{tails_lower_lower}, by
			invoking Proposition \ref{couple}, we equivalently need to show the
			following:

			\begin{theorem}
				\label{tails_lower_lower1}
				For any fixed $\delta_1\in (0,1)$, there exist constants $C,c$
				depending on $\delta_1$ such
			that for all $N\geq N_0$ and for all $y$ such that
			$\delta_1 N^{2/3}>y>0$, we have 
			\begin{displaymath}
				\overline{\mathbb{P}}^{\frac{1}{2}}(G^2_{\mathrm{stat}}(v_N)-N\leq
					-yN^{1/3})\geq Ce^{-cy^3}.
			\end{displaymath}
			\end{theorem}
			
			Note that for the point-to-line representation, we have the
			following analogue of Lemma \ref{expecatation_x^2/n}:

			\begin{lemma}
				\label{expectation}
				If $-\frac{N}{4}<t< \frac{N}{4}$, we have that
				\begin{displaymath}
					f(v_N-(t,-t))=N-4\frac{t^2}{N}-N\mathcal{O}\left(
					(\frac{t}{N})^4
					\right),
				\end{displaymath}
				where the $\mathcal{O}\left(
					(\frac{t}{N})^4
					\right)$ is a term that is strictly positive for all $t$ in the given
					range.
			\end{lemma}
			\begin{proof}
				The proof is a straightforward computation of the Taylor
				expansion of $\left(
				\sqrt{\frac{N}{4}-t}+\sqrt{\frac{N}{4}+t} \right)^2$.
			\end{proof}
			The term $f(v_N-(t,-t))$ in Lemma \ref{expectation} can be
			heuristically thought of as the expected passage time from $0$ to
			$v_N$ if one only maximises over paths which leave the line
			$\{x+y=0\}$ at the point $(t,-t)$. Indeed, there is no term here
			analogous to $g(x)$ in Lemma \ref{expecatation_x^2/n}
			because each random variable on the boundary line $\{x+y=0\}$ has
			mean $0$ for the case $\rho=\frac{1}{2}$. Recall that we use the
			notation $\mathfrak{t}$ for $(-t,t)$. We split the proof of
			Theorem \ref{tails_lower_lower1} into two cases-- the proof for
			the case
			$y\geq \frac{1}{2}N^{1/6}$ is completed in Proposition
			\ref{lower_lower} whereas the case $y<\frac{1}{2} N^{1/6}$
			requires an extra argument.
			\begin{proposition}
				\label{lower_lower}
				For $\delta_1\in (0,1)$, there exist
				positive constants $C_,c,N_0,y_0$ depending on
				$\delta_1$ such that for all $N\geq
				N_0$ and $y_0<y<\delta_1 N^{2/3}$, we have
				\begin{displaymath}
\overline{\mathbb{P}}^{\frac{1}{2}}\left( \max_{t\in
					[-(y^2
					N^{2/3}\wedge
					\frac{N}{4}),y^2
					N^{2/3}\wedge
					\frac{N}{4}]}\{T(\mathfrak{t})+G(
					\mathfrak{t},v_N)\}<N-yN^{1/3}\right)\geq
				Ce^{-cy^3}.
			\end{displaymath}
			\end{proposition}
			\begin{proof} 		
			 Consider the case $y^2 N^{2/3}<\frac{N}{4}$ for now. Recall the
			 notation $G_0(\cdot)$ from the
			 statement after \eqref{point_to_line_defn}. By increasing $y_0$ if necessary, we have the
				following:
			\begin{align}
				&\overline{\mathbb{P}}^{\frac{1}{2}}\left( \max_{t\in
				[-y^2
				N^{2/3},y^2
				N^{2/3}]}\{T(\mathfrak{t})+G(
				\mathfrak{t},v_N)\}<N-yN^{1/3}\right)\nonumber\\
				&\geq
				\overline{\mathbb{P}}^{\frac{1}{2}}\left(\left\{\max_{t\in
				[-y^2
				N^{2/3},y^2
				N^{2/3}]}\left\{T(\mathfrak{t})\right\}<(1-\delta_1)yN^{1/3}\right\}\bigcap\left\{\max_{t\in
				[-y^2
				N^{2/3},y^2
				N^{2/3}]}\left\{G(\mathfrak{t},v_N)\right\}<N-y(2-\delta_1)N^{1/3}\right\}\right)\nonumber\\
					&=
			\overline{\mathbb{P}}^{\frac{1}{2}}\left(\max_{t\in
				[-y^2
				N^{2/3},y^2
				N^{2/3}]}\left\{T(\mathfrak{t})\right\}<(1-\delta_1)yN^{1/3}\right)\overline{\mathbb{P}}^{\frac{1}{2}}\left(\max_{t\in
				[-y^2
				N^{2/3},y^2
				N^{2/3}]}\left\{G(\mathfrak{t},v_N)\right\}<N-y(2-\delta_1)N^{1/3}\right)\nonumber\\
				&\geq
			\overline{\mathbb{P}}^{\frac{1}{2}}\left(\max_{t\in
				[-y^2
				N^{2/3},y^2
				N^{2/3}]}\left\{T(\mathfrak{t})\right\}<(1-\delta_1)yN^{1/3}\right)\mathbb{P}\left(G_0(v_N)<N-y(2-\delta_1)N^{1/3}\right).			\label{lower_tail_lower_bound}
			\end{align}
			The first equality in the above series of expressions follows
			because of the independence of the boundary and
non-boundary random variables. Note that the first term in \eqref{lower_tail_lower_bound} can
			be bounded as follows
			\begin{align}
			\overline{\mathbb{P}}^{\frac{1}{2}}\left(\max_{t\in
				[-y^2
				N^{2/3},y^2
				N^{2/3}]}\left\{T(\mathfrak{t})\right\}<(1-\delta_1)yN^{1/3}\right)&=
			\overline{\mathbb{P}}^{\frac{1}{2}}\left(\max_{t\in
				[-y^2
				N^{2/3},y^2
				N^{2/3}]}\left\{\frac{T(\mathfrak{t})}{\sqrt{y^2N^{2/3}}}\right\}<1-\delta_1\right)\nonumber\\
				&\rightarrow
		\mathbf{P}\left(\max_{t\in[-1,1]}\left\{B_t\right\}<1-\delta_1\right)>c_1>0
		\mathrm{~as~} N\rightarrow \infty.
		\label{first_term_lower_lower}
	\end{align}
	The last line in the above expression follows by Donsker's theorem as
	under $\mathbf{P}$, $B_t$ has the law of the two-sided standard Brownian motion
	started from $0$. Note that the final term is some constant and has no
	dependence on $y$ or $N$, and the estimate is uniform as
	$y>y_0>0$. Finally, to prove the proposition, we need to
	handle the second term in \eqref{lower_tail_lower_bound}, and we will be
	using Proposition \ref{lower_tail_lower_bound_ingredient} with
	$\delta_2=4\delta_1(2-\delta_1)$ for this purpose.
	Indeed, Proposition \ref{lower_tail_lower_bound_ingredient} immediately
	implies the following for all large $N$:
	\begin{align}
		\mathbb{P}\left(G_0(v_N)<N-y(2-\delta_1)N^{1/3}\right) \geq C_2 e^{-c_2 y^3}\label{using_the_ingredient}.
	\end{align}
	Note that we needed to restrict to the case $\rho=\frac{1}{2}$ to use
	Proposition \ref{lower_tail_lower_bound_ingredient}. On combining \eqref{first_term_lower_lower} and
	\eqref{using_the_ingredient} by using \eqref{lower_tail_lower_bound}, we have what we needed, namely
	\begin{displaymath}
		\overline{\mathbb{P}}^{\frac{1}{2}}\left( \max_{t\in
		[-y^2
		N^{2/3},y^2
		N^{2/3}]}\{T(\mathfrak{t})+G(\mathfrak{t},v_N)\}<N-yN^{1/3}\right)\geq
				Ce^{-cy^3}.
	\end{displaymath}
	Note that the case $y^2 N^{2/3}\geq\frac{N}{4}$ follows from the above
	because when we lower bound the term
	\begin{displaymath}
		\overline{\mathbb{P}}^{\frac{1}{2}}\left( \max_{t\in
	[-\frac{N}{4},\frac{N}{4}]}\{T(\mathfrak{t})+G(
				\mathfrak{t},v_N)\}<N-yN^{1/3}\right)
	\end{displaymath}
	by using the method in \eqref{lower_tail_lower_bound}, then the bound in
	\eqref{using_the_ingredient} stays the same, whereas the bound in \eqref{first_term_lower_lower} only becomes better. Indeed for $y^2
	N^{2/3}\geq \frac{N}{4}$, we have
	that	\begin{displaymath}
		\overline{\mathbb{P}}^{\frac{1}{2}}\left(\max_{t\in
		[-\frac{N}{4},\frac{N}{4}]}\left\{T(\mathfrak{t})\right\}<(1-\delta_1)yN^{1/3}\right)\geq \overline{\mathbb{P}}^{\frac{1}{2}}\left(\max_{t\in
		[-\frac{N}{4},\frac{N}{4}]}\left\{\frac{T(\mathfrak{t})}{\sqrt{N/4}}\right\}<1-\delta_1\right),
	\end{displaymath}
	and the result follows by using Donsker's theorem as in
	\eqref{first_term_lower_lower}.
\end{proof}
\begin{remark}
	\label{difficulties}
	 The same approach does not work directly for general $\rho$ because for
	 $\rho\neq \frac{1}{2}$, the weight function $T$ on the line
	 $\{x+y=0\}$ in the point-to-line representation is a random walk with non-zero drift which causes the
	 Brownian estimate in \eqref{lower_tail_lower_bound} to no longer
	 work. To use the same
	 approach for general $\rho$, one would need to obtain a point-to-line estimate,
	 similar to Proposition \ref{lower_tail_lower_bound_ingredient}, adjusted
	 to included the drift term on the line.
\end{remark}
We now complete the proof of Theorem \ref{tails_lower_lower1}.

\begin{proof}[Proof of Theorem \ref{tails_lower_lower1}]
	Note that we already proved the result for the case
	$y^2\geq \frac{N^{1/3}}{4}$ in Proposition \ref{lower_lower}, so we
	restrict to the case $y^2<\frac{N^{1/3}}{4}$. Also, the
	case $0<y\leq y_0$ can be handled by adjusting $C,c$. For the
	point-to-line representation, let $A_1$ denote the event that the
	stationary geodesic to $v_N$ does not intersect the line segment
	$\{\mathfrak{t}: t\in [-y^2 N^{2/3},y^2 N^{2/3}]\}$. Similarly, let $A_2$ the
	event that the stationary geodesic to $v_N$ does not intersect the
	region $\{\mathbf{x}:x\in [-y^2N^{2/3},y^2 N^{2/3}]\}$. Observe that we
	have
	\begin{align}
	\overline{\mathbb{P}}^{\frac{1}{2}}(G^2_{\mathrm{stat}}(v_N)-N\leq
					-yN^{1/3})&\geq  \overline{\mathbb{P}}^{\frac{1}{2}}\left( \max_{t\in
					[-y^2
					N^{2/3},y^2
					N^{2/3}]}\{T(\mathfrak{t})+G(
					\mathfrak{t},v_N)\}<N-yN^{1/3}\right)
					-\overline{\mathbb{P}}^{\frac{1}{2}}(A_1)\nonumber\\
					&\geq C_1e^{-c_1y^3}
					-\overline{\mathbb{P}}^{\frac{1}{2}}(A_2)\nonumber\\
					&=C_1e^{-c_1y^3}-\mathbb{P}^{\frac{1}{2}}\left(
					|Z^{\mathbf{0}\rightarrow v_N}|\geq y^2 N^{2/3}
					\right)\nonumber\\
					&\geq C_1e^{-c_1y^3} - C_2 e^{-c_2 y^6}\nonumber\\
					&\geq C e^{-c y^3}
		\end{align}
		The first inequality is a simple union bound. The second inequality
		follows by observing that $A_1\subseteq A_2$ along with applying
		Proposition \ref{lower_lower}. The first equality follows by
		applying Proposition \ref{geodesic_coupling}, and the third
		inequality in an immediate application of Theorem \ref{exit_time}.
\end{proof}

\section{Appendix}
We first give the proof of Lemma
\ref{mgf_converges} and Lemma \ref{mgf_converges1}.
		\begin{proof}[Proof of Lemma \ref{mgf_converges}]
				Using the explicit form of the moment generating function for
				the exponential distribution, we have that
				\begin{displaymath}
					\mathbb{E}^{\rho}\left[
			\exp\left(\frac{\lambda
			(G^1_{\mathrm{stat}}(\mathbf{\overline{x}})-\frac{\overline{x}}{1-\rho})}{\frac{\sqrt{r+1}N^{1/3}}{1-\rho}}\right)\right]
			 =\left(
			 \frac{1}{1-\frac{\lambda}{\sqrt{\overline{x}}}}e^{-\frac{\lambda}{\sqrt{\overline{x}}}}
			\right)^{\overline{x}}.
				\end{displaymath}
				Now, just note that there is a constant $\delta_0\in (0,1)$ such that the function $\left(
				\frac{1}{1-\frac{y}{n}}e^{-\frac{y}{n}}
			\right)^{n^2}$ is at most $e^{C^*y^2}$ if
			$\frac{y}{n}<\delta_0$. This can be seen right away by taking a logarithm and doing
			a Taylor expansion.
	\end{proof}

	\begin{proof}[Proof of Lemma \ref{mgf_converges1}]
				We will directly use Lemma \ref{mgf_converges}. Note
				that
				$\frac{\lambda}{\sqrt{\overline{x}}}=\frac{C_1}{C^*}\frac{r^{3/2}}{\sqrt{r+1}N^{1/3}}\leq
				\frac{C_1}{C^*}\frac{r}{N^{1/3}}$, where $C_1$ depends only on
				$\rho$. Hence, by Lemma \ref{mgf_converges}, we just require
				that $\frac{C_1}{C^*}\gamma_1\leq \delta_0$ which can
				be arranged by choosing $C^*$ large enough because $\delta_0\rightarrow 1$ as
				$C^*\rightarrow \infty$.
			\end{proof}
	Recall the notation $\vv{n}$ for $(n,n)$. We now describe the coupling which implies Proposition \ref{couple} and
	Proposition \ref{geodesic_coupling}.

	\begin{proposition}
		\label{coupling_coupling}
		For each fixed $n>0$, there exists a coupling between the two
		stationary representations such that for each $v\in
		\mathbb{Z}^2_{\geq 0}$
		with $v\leq \vv{n}$, we have 
		$G^1_{\mathrm{stat}}(v)=G^2_{\mathrm{stat}}(v)$. 
	\end{proposition}

	\begin{proof}
		The coupling proceeds by a version of the Burke property for the boundary representation of stationary LPP (Lemma 4.2 in
	\cite{BCS06}). To be precise,
	construct a boundary stationary model with its origin shifted to the
	point $(-n,-n)$; call the associated environment on
	$-\vv{n}+\mathbb{Z}^2_{\geq 0}$ as $\nu$, and let $G^1_{\mathrm{stat}}(
	-\vv{n},\cdot)$
	denote the associated stationary passage times. From $\nu$, derive the
	environment $\nu'$ on $\mathbb{Z}^2_{\geq 0}$ as follows:
	\begin{displaymath}
		\nu'_v=
		\begin{cases}
			\nu_v, & \text{ for } v\in \mathbb{Z}^2_{>0}\\
			G^1_{\mathrm{stat}}\left( -\vv{n},i\mathrm{e}_1 \right)-G^1_{\mathrm{stat}}\left(
			-\vv{n},(i-1)\mathrm{e}_1 \right), & \text{ for }
			v=i\mathrm{e}_1\\
				G^1_{\mathrm{stat}}\left( -\vv{n},j\mathrm{e}_2 \right)-G^1_{\mathrm{stat}}\left(
			-\vv{n},(j-1)\mathrm{e}_2 \right), & \text{ for }
			v=j\mathrm{e}_2\\
			0 \text{ for } v=\mathbf{0}.
		\end{cases}
	\end{displaymath}
	Use $\nu'$ to define a boundary stationary model on
	$\mathbb{Z}^2_{\geq 0}$, and use $G^1_{\mathrm{stat}}(\cdot)$ for the
	associated stationary passage times. We now derive a point-to-line stationary model
	from $\nu$. Recall that we only want to define $G^2_{\mathrm{stat}}(v)$
	for $v\leq \vv{n}$ and we need to define the environment accordingly.
	Define the environment $\nu''$ to be the same as $\nu$ in region
	$\{x+y> 0\} \cap \{(x,y)\leq \vv{n}\}$. Define the weight function $T$
	on $\{x+y=0\}\cap \{(x,y)\leq \vv{n}\}$ as
	follows:
	\begin{displaymath}
		T(\mathfrak{t})=
		G^1_{\mathrm{stat}}(-\vv{n},\mathfrak{t})-G^1_{\mathrm{stat}}(-\vv{n},\mathbf{0}).
	\end{displaymath}
	Use $T$ and $\nu''$ to construct the point-to-line representation as in
	\eqref{repn_pt_line} and thus define $G^2_{\mathrm{stat}}(v)$ for all $0\leq v\leq n$. By an
	application of Lemma 4.2 in \cite{BCS06}, the quantities
	$G^1_{\mathrm{stat}}(\cdot)$ and $G^2_{\mathrm{stat}}(\cdot)$ have the
	correct marginal distributions. It is easy to see that because of the
	coupling, we have (cf. Lemma 3.3 and Lemma 3.4 from \cite{SS19})
	\begin{equation}
		\label{coupling}
		G^1_{\mathrm{stat}}(v)=G^2_{\mathrm{stat}}(v)=
	G^1_{\mathrm{stat}}\left( -\vv{n},v \right) -G^1_{\mathrm{stat}}\left(
	-\vv{n},\mathbf{0} \right).
	\end{equation}
	for all $0\leq v \leq \vv{n}$.
	\end{proof}
	\begin{proof}[Proof of Proposition \ref{couple}]
		Immediate from Proposition \ref{coupling_coupling}.
	\end{proof}
	\begin{proof}[Proof of Propositon \ref{geodesic_coupling}]
		By the
		almost sure uniqueness of geodesics, for either representation, the stationary geodesic to
		$v_N$ (restricted to $\mathbb{Z}^2_{>0}$) can be
		reconstructed just by knowing the stationary passage times for all $v$
		with
		$\mathbf{0}\leq v\leq v_N$. For example,
		working with the point-to-line
		representation, if we have
		$G^2_{\mathrm{stat}}(v_N-\mathrm{e}_1)>G^2_{\mathrm{stat}}(v_N-\mathrm{e}_2)$,
		then the stationary geodesic to $v_N$ is just the stationary
		geodesic till $v_N-\mathrm{e}_1$ concatenated with the singleton
		$\{v_N\}$,
		and one can proceed recursively to obtain the portion of the
		stationary geodesic in $\mathbb{Z}^2_{>0}$. By Proposition
		\ref{coupling_coupling}, there is a coupling for which
		$G^1_{\text{stat}}(v)=G^2_{\text{stat}}(v)$ for all $\mathbf{0}\leq
		v\leq v_N$. By the above discussion, this implies that the
		stationary geodesics till $v_N$ agree in the region
		$\mathbb{Z}^2_{>0}$ for both the representations and Proposition
		\ref{geodesic_coupling} follows immediately.
	\end{proof}
\end{document}